\documentclass[draft,twoside,10pt,leqno]{amsart}
\usepackage{srcltx}
\usepackage[english]{babel}
\numberwithin{equation}{section}
\newtheorem{thm}{Theorem}[section]

\newtheorem*{thms}{Theorem}
\newtheorem*{thmA}{Theorem A}
\newtheorem*{corB}{Corollary B}
\newtheorem*{thmC}{Theorem C}
\newtheorem*{thmA'}{Theorem A'}
\newtheorem*{thmB'}{Theorem B'}

\newtheorem*{thmA''}{Theorem A''}
\newtheorem*{thmB''}{Theorem B''}
\newtheorem{lem}[thm]{Lemma}

\newtheorem{cor}[thm]{Corollary}

\theoremstyle{definition}
\newtheorem{defi}[thm]{Definition}
\newtheorem*{defis}{Definition}
\newtheorem{exe}[thm]{Example}

\newtheorem{rmk}[thm]{Remark}

\numberwithin{equation}{section}
\usepackage{amssymb,amsmath, array}
\usepackage{mathtools}

\def\div{\mathop\mathrm{div}\nolimits}

\def\tr{\mathop\mathrm{tr}\nolimits}

\def\vol{\mathop\mathrm{vol}\nolimits}
\def\inn{\mathop\mathrm{int}\nolimits}
\def\inj{\mathop\mathrm{inj}\nolimits}
\def\dist{\mathop\mathrm{dist}\nolimits}

\def\cut{\mathop\mathrm{cut}\nolimits}

\newcommand{\Hess}{\operatorname{Hess}}

\newcommand{\e}{\mbox{$\mathrm{e}$}}
\newcommand{\R}{\mbox{${\mathbb R}$}}

\newcommand{\g}[2]{\mbox{$\langle #1 ,#2 \rangle$}}
\newcommand{\fle}{\mbox{$\rightarrow$}}
\newcommand{\ep}{\varepsilon}
\newcommand{\rf}[1]{\mbox{(\ref{#1})}}
\newcommand{\rl}[1]{{~\ref{#1}}}

\newcommand{\M}{(M,\g{\ }{\ })}
\newcommand{\dM}{(M,\partial M,\g{\ }{\ })}
\newcommand{\dMc}{(M,\partial M,\widetilde{\g{\ }{\ }})}
\newcommand{\uF}{\mathcal{F}(M)}
\newcommand{\uA}{\mathcal{B}_2(M)}
\newcommand{\uB}{\mathcal{B}_1(M)}

\def\beq{\begin{equation}}
\def\eeq{\end{equation}}
\def\beqs{\begin{equation*}}
\def\eeqs{\end{equation*}}

\begin{document}

\title[Schwarz lemma for noncompact manifolds with boundary]{A Schwarz-type lemma\\ for noncompact manifolds with boundary\\
and geometric applications}

\author[Guglielmo Albanese]{Guglielmo Albanese}
\author[Marco Rigoli]{Marco Rigoli}

\address{Dipartimento di Matematica, Universit\`{a} degli Studi di Milano, Via Saldini 50, I-20133, Milano, Italy}
\email {guglielmo.albanese@unimi.it}
\email {marco.rigoli@unimi.it}

\address{Departamento de Matem\'atica, Universidade Federal do Cear\'a
Av. Humberto Monte s/n, Bloco 914, 60455-760 Fortaleza, Brazil}
\email {marco.rigoli55@gmail.it}



\date{\today}

\maketitle
\begin{abstract}
We prove a Schwarz-type lemma for noncompact manifolds with possibly noncompact boundary. The result is a consequence of a suitable form of the weak maximum principle of independent interest. The paper is enriched with applications to conformal deformations of noncompact manifolds with boundary, among them a generalization of a classical result by Escobar.
\end{abstract}

\section{Introduction}
The Schwarz Lemma (see III.3.I in \cite{Ca}) is a basic tool in complex analysis whose importance can be hardly overstimated; its use for a one-line-proof of Liouville's theorem on constancy of entire holomorphic functions is an enlightening example of its strength. As reported in detail by R. Osserman in his survey \cite{O1}, beside its use in complex analysis, the Schwarz Lemma turns out to be a fundamental tool in studying properties of conformal deformations of manifolds of negative curvature. The main observation that leads to this use of the result is the \emph{geometric} formulation of the Lemma, namely the so called Schwarz-Pick Lemma proved by G. Pick in \cite{Pi}. We recall that the Schwarz-Pick Lemma states that if $f(z)$ is a holomorphic map from the unit disk $D$ into itself, then
	\beq\label{distH}
	\dist_{\mathbb{H}}(f(z_1),f(z_2))\leq\dist_{\mathbb{H}}(z_1,z_2)\quad\quad\hbox{for all $z_1, z_2\in D$,}
	\eeq
where $\dist_{\mathbb{H}}$ denotes the hyperbolic distance in $D$. In other words, a holomorphic map from the unit disk into itself decreases the hyperbolic distance.\\
Beside the interest of this result concerning geometric function theory, the lemma had to become a key to open the door of the theory of holomorphic maps between Riemannian manifolds. Indeed, in 1938 L.V. Ahlfors generalized the Schwarz-Pick Lemma considering holomorphic maps from the unit disk $D$ into a general Riemann surface of negative curvature \cite{Ah}. After this seminal paper, many efforts have been made to deal with maps from general Riemann surfaces and, more generally, with maps between higher dimensional complex manifolds. A further step in generalizing the result is not to consider just a holomorphic map from a complex manifold to another, but instead to deal with conformal mappings between Riemannian manifolds (holomorphic maps are conformal). In these directions the literature is wide and we only cite the cornerstone papers by S.T. Yau \cite{YaSc, YaKa}, the well known book of S. Kobayashi on hyperbolic complex spaces \cite{Ko}, and a more recent paper by A. Ratto, M. Rigoli, L. V\'eron \cite{RRV}.\\
\newline

In this paper we deal with the case of pointwise conformal deformations of noncompact Riemannian manifolds with boundary. It seems that this case has not been considered previously in the literature, indeed, the research that stemmed from the Schwarz-Pick-Ahlfors Lemma has been focused on the complete and boundaryless case. An intriguing feature of considering the case of manifolds with boundary is that it resembles the results of K. L\"owner and J. Velling about holomorphic mappings between disks. In particular we recall the boundary Schwarz lemmas by D. Burns and S. Krantz \cite{BK}, and R. Osserman \cite{O3}. We refer to the recent surveys by H. Boas \cite{Bo} and S. Krantz \cite{Kr} for a comprehensive treatment of the boundary Schwarz Lemma.\\
In our investigation we need to introduce some technical tools which are interesting in their own. More specifically, in Section 3 we generalize the weak maximum principle (see for instance \cite{AAR}) to noncompact manifolds with boundary. In doing so we introduce two different function spaces for which we have the validity of the present version of the weak maximum principle. We then apply these results to obtain not only a generalization of the Schwarz lemma, but also some rigidity results concerning conformal diffeomorphisms of noncompact manifolds. Towards this last goal we need to provide an $\mathrm{L}^{\infty}$ estimate for solutions of certain differential inequalities which naturally appears in this and other relevant geometric contexts.\\ 
\newline

From now on we suppose that $\dM$ is a smooth Riemannian manifold with smooth boundary $\partial M$ and $m=\dim M\geq3$. An origin $o$ is fixed for the manifold, $r:M\rightarrow\R^+_0$ denotes the distance from $o$, namely $r(x)=\dist(x,o)$. The ball of radius $R$ with respect to this distance is denoted by $B_R$. We recall that a pointwise conformal deformation of $\dM$ is the Riemannian manifold $\dMc$ where $\widetilde{\g{\ }{\ }}=u^{\frac{4}{m-2}}\g{\ }{\ }$ for some smooth positive function $u$ called the conformal factor of the deformation. We denote with $(s,\, h)$ and $(\widetilde{s},\, \widetilde{h})$ the scalar curvature and the mean curvature of the boundary respectively of $\dM$ and $\dMc$, where $h$ is the mean curvature of the boundary with respect to the unit outward normal (note that with this convention the boundary of the Euclidean ball has negative mean curvature). Then, as it is well known (see for instance \cite{Ch,Es}), these quantities are related as follows
	\beq
		\begin{dcases}\label{bYa}
		\Delta u-c_m\left(s(x)u-\widetilde{s}(x)u^{\frac{m+2}{m-2}}\right)=0 & \hbox{on $M$}\\
		\partial_{\nu}u+d_m\left(h(x)u-\widetilde{h}(x)u^{\frac{m}{m-2}}\right)=0 & \hbox{on $\partial M$} 
		\end{dcases}
	\eeq
where $\Delta$ and $\nu$ are the Laplace-Beltrami operator of $M$ and the outward unit normal of $\partial M$ in the background metric $\g{\,}{\,}$, while $c_m$ and $d_m$ are the constants respectively given by
	\beqs
		\begin{aligned}
		c_m=\frac{m-2}{4(m-1)}\,,&\quad & d_m=\frac{m-2}{2}\,.
		\end{aligned}
	\eeqs
The problem of finding a pointwise conformal deformation of $\dM$ with prescribed scalar curvature in $M$ and prescribed mean curvature of $\partial M$ has been first considered by Cherrier in \cite{Ch}. A few years later in two cornerstone papers \cite{Es,Esann} Escobar considered the related Yamabe problem on compact manifolds with boundary. Since then, many efforts have been made towards a complete solution of the boundary Yamabe problem in the compact case.
To the best of our knowledge the first who settled it in the case of noncompact manifolds with boundary has been F. Schwartz in \cite{Sc}. In this paper he considers the problem of finding a conformal diffeomormorphism with $\widetilde{s}\equiv 0$ and prescribed $\widetilde{h}$ on a noncompact manifold with compact boundary and a controlled volume growth on each end. Another related work is the very recent paper by Almaraz at al. \cite{ABdL} where they consider a positive mass theorem for asymptotically flat manifolds with noncompact boundary.
We tackle the problem of prescribing the scalar curvature in the more general case of a noncompact manifold with possibly noncompact boundary.\\

Following the philosophy of Pick, our generalization of Schwarz Lemma is stated, as in \rf{distH}, in terms of contraction of distances. Let us recall that a conformal diffeomorphism $f:\dM\rightarrow\dM$ with conformal factor $u$ is said to be weakly distance decreasing if $u\leq 1$ on $M$, see \cite{RRV}. Our main result is then the following
	\begin{thmA}\label{sclem}
	Let $\dM$ be a complete, noncompact, Riemannian manifold with boundary $\partial M$. Assume that
		\beq\label{volQ}
		\liminf_{t\rightarrow+\infty}\frac{Q(t)\log\vol B_t}{t^2}<+\infty
		\eeq
	where $Q(t)$ is a nondecreasing function satisfying $Q(t)=o(t^2)$ as $t\rightarrow+\infty$. Let $f$ be a conformal diffeomorphism of $\dM$ into itself such that, for some constant $c>0$, the scalar curvature $\widetilde{s}(x)$ of the new metric $\widetilde{\g{\,}{\,}}=f^*\g{\,}{\,}=u^{\frac{4}{m-2}}\g{\,}{\,}$ satisfies 
		\beqs
		-c\leq\widetilde{s}(x)<\min\left\{0,s(x)\right\}\quad\hbox{on $M$}
		\eeqs
	and
		\beqs
		\widetilde{s}(x)\leq-\frac{1}{Q(r(x))}\quad\hbox{outside a compact set}\,.
		\eeqs
	Furthermore, for $\gamma\in\R$ let 
		\beqs
		\Omega_{\gamma}=\left\{x\in M\,:\,u(x)>\gamma\right\}
		\eeqs
	and assume that
		\beq\label{hsc}
		\widetilde{h}(x)\leq u^{-\frac{2}{m-2}}h(x)\quad\hbox{on $\overline{\Omega}_{\gamma}\cap\partial M$}
		\eeq		
	for some $\gamma<u^*\leq+\infty$.
	Then $f$ is weakly distance decreasing.
	\end{thmA}
	
We remark that, although we stated our result when the domain and target manifolds coincide, it can be easily generalized to the case of a conformal map between different manifolds with boundary. This result basically extends Theorem 3.3 of \cite{PRSvol} to this new setting. The delicate issue in the present case is due to condition (\ref{hsc}) which involves the conformal factor $u$; however (\ref{hsc}) is satisfied with no reference to $u$ whenever the geometric requirement
	\beqs
	\widetilde{h}(x)\leq 0\leq h(x)\quad\hbox{on $\overline{\Omega}_{\gamma}\cap\partial M$}\,,
	\eeqs
holds.
In view of applications it is meaningful to introduce the following
	\begin{defis}
	Let $\dM$ be a Riemannian manifold with boundary and dimension $m\geq3$. We say that a conformal diffeomorphism $f$ of $M$ with conformal factor $u$ is \emph{$\partial$-rigid} if 
		\beqs 
		\partial_{\nu} u=0 \quad \hbox{on $\partial M$}.
		\eeqs
	\end{defis}
With this definition in mind we obtain the following corollary of Theorem \ref{sclem} below which characterizes isometries in the class of conformal diffeomorphisms of $\dM$ into itself preserving the scalar curvature. This is in the vein of the investigation program inspired by Yau's paper \cite{YaSc} (in particular Corollary 1.2).
	\begin{corB}
	
	Let $\dM$ be a complete, noncompact, manifold with boundary, dimension $m\geq 3$ and scalar curvature $s(x)$ satisfying
		\beq\label{conds}
		i)\,  -c \leq s(x) < 0\,, \quad ii)\, s(x) \leq -\frac{1}{Q(r(x))} \quad\hbox{for $r(x)>> 1$}
		\eeq
	for some positive constant $c$ and with $Q(t)$ as in the statement of Theorem A. Assume that (\ref{volQ}) holds. Then, any conformal diffeomorphism $f$ of $\dM$ into itself which is $\partial$-rigid and preserves the scalar curvature is an isometry.
	\end{corB} 
	
We turn our attention to a slightly different geometric problem proposed by Escobar in the compact case \cite{Es2}. The precise question is: given a conformal diffeomorphism of a Riemannian manifold with boundary $\dM$ such that $\widetilde{s}=s$ on $M$ and $\widetilde{h}=h$ on $\partial M$, when does it hold true that $\widetilde{\g{\ }{\ }}=\g{\ }{\ }$? He proved the following 
	\begin{thms}[Corollary 2 in \cite{Es2}]
	Let $\M$ be a compact Riemannian manifold with boundary. Assume that $\widetilde{\g{\ }{\ }}=u^{\frac{4}{m-2}}\g{\ }{\ }$, $\widetilde{s}=s\leq 0$ on $M$ and $\widetilde{h}=h\leq 0$ on $\partial M$. Then $\widetilde{\g{\ }{\ }}=\g{\ }{\ }$.
	\end{thms}
For the noncompact case we have an analogous rigidity result, namely
	\begin{thmC}
	Let $\dM$ be a complete, noncompact, manifold with boundary, dimension $m\geq 3$ and scalar curvature $s(x)$ satisfying (\ref{conds}) for a positive constant $c$. Assume that (\ref{volQ}) holds. 
	Then the identity is the only conformal diffeomorphism of $\dM$ into itself such that $\widetilde{s}=s$ on $M$ and $\widetilde{h}=h\leq 0$ on $\partial M$.
	\end{thmC}
We stress the fact that Theorem C has the same hypotheses of the theorem by Escobar, without any other technical assumption. It is just necessary to control the growth of geodesic balls at infinity.\\
\newline

All the above results are proved with the aid of a suitable form of the weak maximum principle for manifolds with boundary. The structure of the paper is as follows.\\
In Section \ref{secbdy} we report just the basic facts about the geometry of manifolds with boundary.
In Section \ref{secwmp} we develope the form of the weak maximum principle adequate for our aims.
Section \ref{secconf} is devoted to the proof of the main results of the paper and other related geometric applications.

\section{Complete Riemannian manifolds with boundary}\label{secbdy}
In this section we fix notations and collect some useful facts on the geometry of complete Riemannian manifolds with boundary. From now on $\dM$ will denote a smooth, complete Riemannian manifold of dimension $m\geq2$, and smooth boundary $\partial M$.\\ 
It is worth to spend some words on the notion of completeness for a Riemannian manifold with boundary. Indeed in this case the familiar Hopf-Rinow theorem does not hold, because the presence of the boundary prevents the infinite extendability of geodesics. Thus the completeness of $\dM$ has to be understood in the sense of the metric spaces. Here the distance between two points $p,q\in M$ is defined as usual as
	\beqs
	\dist(p,q)=\inf_{\sigma\in\Sigma^1_{p,q}}l(\sigma)
	\eeqs
where $\Sigma^1_{p,q}$ is the space of $\mathrm{C}^1$ paths starting at $p$ and ending at $q$, and $l(\sigma)$ is the length of $\sigma$ with respect to the metric $\g{\,}{\,}$. The first awkward thing to be noted is that, differently to what happens when the boundary is empty, the optimal regularity of a geodesic between $p$ and $q$ is $C^{1,1}$ even if the boundary is smooth. For a deep analysis of the situation we refer to a series of papers by S.B. Alexander, I.D. Berg, and R.L. Bishop \cite{ABB1,ABB2,ABB3}.\\
In the sequel we will assume that a reference point $o\in M$ has been fixed and we will denote by $r:M\rightarrow\R^+_0$ the distance function from $o$, that is,
	\beqs
	r(x)=\dist(o,x)\,,
	\eeqs
clearly $r\in\mathrm{Lip}(M)$. Moreover, for $t\in\R^+$ and $y\in M$, we let $B_t(y)$ be the geodesic ball of radius $t\in\R^+$ centered at $y\in M$, that is,
	\beqs
	B_t(y):=\left\{x\in M: \dist(x,y)<t\right\}\,,
	\eeqs
in particular we set $B_t=B_t(o)$.\\
Let $\rho:M\rightarrow\R_0^+$ be the distance function from the boundary defined as 
	\beqs
	\rho(x)=\dist(x,\partial M)=\inf_{y\in \partial M}\dist(x,y)\,,	
	\eeqs
where the infimum is always attained since $\partial M$ is a closed set of a complete metric space. Moreover $\rho\in\mathrm{Lip}(M)$ and it is smooth and minimizes the distance from $\partial M$ out of his cut locus, which is a set of measure zero (see for instance \cite{MM}). For $\ep>0$ we set 
	\beqs
	M_{\ep}=\left\{x\in M: \rho(x)<\ep\right\}\,.
	\eeqs 
We introduce the Fermi coordinates with respect to the boundary $\partial M$ (see for instance Section 10 of \cite{PR} for a well written review of Fermi coordinates). Let us define, for $y\in\partial M$ and $t\in\R^+$,
	\beqs
	\Phi_{\partial M}(y,t):=\exp_y(-t\nu_y),
	\eeqs
where $\exp$ denotes the exponential map and $\nu_y$ the outward normal at the point $y$.
From the properties of $\rho$ (see \cite{MM}), for each $y\in\partial M$ there exist $\ep_y>0$ such that for $t\in[0,\ep_y)$, $\Phi_{\partial M}(y,t)$ does not meet the cut locus of $y$, we define $\tau_y$ to be the $\sup$ of these $\ep_y$. In general, if $\partial M$ is noncompact, it can happen that $\inf_{y\in\partial M}\tau_y=0$, this implies that it could not exist an $\ep$ such that there exist global Fermi coordinates on $M_{\ep}$. Let $U=U_y\subset\partial M$ be a open and bounded set (in the topology of $\partial M$), set $\tau_U=\inf_{y\in U}\tau_y$ and let $0<\tau<\tau_U$, then we define the \emph{Fermi cylinder} of base $U$ and height $\tau$
	\beq\label{fermi}
	C_{y}(U,\tau):=\left\{\Phi_{\partial M}(y,t):y\in U, t\in[0,\tau)\right\}\,.
	\eeq

\section{A weak maximum principle for manifolds with boundary}\label{secwmp}

In what follows $q(x)$ will denote a continous and positive function on $M$. Let $\uF$ be a set of functions defined on $M$ such that $\mathrm{C}^0(M)\subseteq\uF$. We start by stating the following
	\begin{defi}\label{dWMP}
	Let $\dM$ be a complete Riemannian manifold with non-empty boundary. We say that a function $u\in\uF$ such that $u^*=\sup_{M}u<+\infty$, satisfies the \emph{q-boundary weak maximum principle}, for short $q$-$\partial WMP$, on $M$ for the operator $L$ if for each $\gamma<u^*$ we have
		\beqs
		\inf_{\Omega_{\gamma}}q(x)L u\leq 0\, ,
		\eeqs
	where $\Omega_{\gamma}$ denotes the superlevel set
		\beqs
		\Omega_{\gamma}=\left\{x\in M\,:\,u(x)>\gamma\right\}.
		\eeqs
	\end{defi}
	
The definition above extends the corresponding definition of the weak maximum principle by Pigola, Rigoli, and Setti \cite{PRS} (see also the very recent improvements in \cite{AAR,AMR} and the book \cite{AMaR}) to the case of manifolds with boundary. We note that recently some attention has been put on global properties of solutions (or subsolutions) to elliptic equations on complete manifolds with boundary (see for instance \cite{IPS}).

Moreover here the point is to put emphasis on the importance of the choice of a suitable functional space $\uF$ to obtain the validity of the maximum principle. Although this point of view could seem artificial, it will be apparent in the sequel that  the presence of a possibly nonempty boundary $\partial M$ generates some subtleties. The following example suggests the necessity of some boundary conditions for the validity of the weak maximum principle.
	\begin{exe}
	For some fixed $\ep>0$, we define the subset of $\R^m$
		\beqs 
		\Lambda=\left\{x=(x_1, \dots , x_m)\in\R^m\,:\,x_m-\sum_{i=1}^{m-1}x_i^2\geq\ep^2\right\}\,,
		\eeqs
	clearly $\Lambda$ is a complete Riemannian manifold with boundary. Consider the function
		\beqs
		u(x)=\ep-\left(x_m-\sum_{i=1}^{m-1}x_i^2\right)^{1/2}\,,
		\eeqs
	it is easy to see that $u\in\mathrm{C}^1(\Lambda)\cap\mathrm{C}^{\infty}(\inn\Lambda)$. Furthermore $u\leq 0$ on $\Lambda$, the maximum $u^*=0$ is attained at each point of $\partial\Lambda$ and only there. Indeed, for $\gamma<0=u^*$ the superlevel set $\Omega_{\gamma}$ is given by
		\beqs
		\Omega_{\gamma}=\left\{x\in\R^m\,: \ep^2\leq x_m-\sum_{i=1}^{m-1}x_i^2\leq\left(\ep-\gamma\right)^2\right\}\,.
		\eeqs
	A simple computation yields	
		\beqs
		\Delta u=\frac{1}{4\left(x_m-\sum_{i=1}^{m-1}x_i^2\right)^{3/2}}+\frac{(m-1)x_m+(2-m)\sum_{i=1}^{m-1}x_i^2}{\left(x_m-\sum_{i=1}^{m-1}x_i^2\right)^{3/2}}\,,
		\eeqs
	from which it follows that 
		\beqs
		\inf_{\Omega_{\gamma}}\Delta u=\frac{1+4(m-1)\ep^2}{(\ep-\gamma)^3}>0\,.
		\eeqs
	\end{exe}
	We note that in the example above the function $u$ is such that
		\[
		\partial_{\nu}u>0\quad\quad\hbox{on $\partial\Lambda$}\,.
		\]
	This shows that in general, we cannot expect to have the validity of the weak maximum principle if the outer normal derivative on the boundary is positive. On the other hand we will prove that, requiring a suitable relaxed form of the inequality 
		\[
		\partial_{\nu}u\leq 0\quad\quad\hbox{on $\partial\Lambda$}\,,
		\]
	the weak maximum principle holds true.\\

	In what follows we shall deal with a large class of linear operators that we are now going to define. We let $T$ be a symmetric, positive definite, covariant $2$-tensor field on $M$. We define the operator $L=L_{T}$ acting on $u\in\mathrm{C}^2(M)$ as
	\beq\label{defL}
	\begin{aligned}
	Lu & =\div\left(\widetilde{T}(\nabla u)\right)\\
	& =\tr\left(\widetilde{T}\circ\widetilde{\Hess}(u)\right)+\div T(\nabla u)
	\end{aligned}
	\eeq
where $\widetilde{T}$ and $\widetilde{\Hess}(u)$ are the symmetric endomorphisms naturally associated to $T$ and $\Hess(u)$. Of course on a manifold with boundary $\dM$ differential inequalities related to the above operator can be interpreted in the following weak sense: $u\in\mathrm{C}^2(M)$ is a solution of the differential inequality
	\beqs
	Lu\geq f(u)
	\eeqs
for some $f\in \mathrm{C}^0(\R)$, if and only if $\forall\,\phi\in \mathrm{C}^{\infty}_c(M)$, $\phi\geq 0$
	\beq\label{weakL}
	\int_M\left[T(\nabla u,\nabla\phi)+\phi f(u)\right]\leq\int_{\partial M}\phi\,T(\nabla u,\nu)\,
	\eeq
where $\nu$ is the outward unit normal to $\partial M$. Moreover, the validity of inequality
	\beq\label{weakL2}
	\int_M\left[T(\nabla u,\nabla\phi)+\phi f(u)\right]\leq 0\,
	\eeq
for all $\phi\in \mathrm{C}^{\infty}_c(M)$, $\phi\geq 0$ defines a weak solution of the Neumann problem
	\beq\label{dweakL}
	\begin{dcases}
	Lu\geq f(u) & \hbox{on $M$}\\
	T(\nabla u,\nu)\leq0 & \hbox{on $\partial M$}\,.
	\end{dcases}
	\eeq 
The key point here is that we will exploit the weak form (\ref{dweakL}) to extend the action of (\ref{defL}) to broader classes of functions than $\mathrm{C}^2(M)$. 
Indeed we observe that H\"older's inequality implies that given $\phi\in \mathrm{C}^{\infty}_c(M)$, $\phi\geq 0$, the left hand side of (\ref{dweakL}) is well defined for any $u\in\mathrm{C}^{0}(M)\cap \mathrm{W}^{1,2}_{\mathrm{loc}}(M)$ (indeed $u\in\mathrm{L}^{\infty}_{\mathrm{loc}}(M)\cap \mathrm{W}^{1,2}_{\mathrm{loc}}(M)$ would be sufficient). 
When $\partial M\neq\emptyset$, the interpretation of right hand side of (\ref{dweakL}) requires a more subtle analysis. Here the issue is that the boundary $\partial M$ is a set of measure zero in $M$ and this means that the integral
	\beqs
	\int_{\partial M}\phi\,T(\nabla u,\nu)
	\eeqs
in general is not well defined for $u\in\mathrm{C}^{0}(M)\cap \mathrm{W}^{1,2}_{\mathrm{loc}}(M)$.

A first way to solve the problem is to use the trace theorem  of Gagliardo (see for instance Theorem 4.12 of \cite{AF}) which ensures that functions in $\mathrm{W}^{2,2}_{\mathrm{loc}}(M)$ have a well defined trace
	\beqs
	T(\nabla u,\nu)\in\mathrm{L}^{2}_{\mathrm{loc}}(\partial M)\,.
	\eeqs
Another way is to restrict the test functions to $\phi\in \mathrm{C}^{\infty}_c(M)$, $\phi\geq 0$, and such that $\left.\phi\right|_{\partial M}\equiv 0$. In this way the boundary term vanishes identically.\\
By a standard density argument in the discussion above it is equivalent to take as test functions $\phi\in\mathrm{W}^{1,2}_{0}(M)$, $\psi\geq 0$. Here as usual $\mathrm{W}^{1,2}_0(M)$ denotes the closure of $\mathrm{C}^{\infty}_c(M)$ with respect to the $\mathrm{W}^{1,2}$-norm.\\

The first result (Theorem \ref{suffvol} below) gives a useful criterion for the validity of $q$-$\partial$WMP for the operator $L_{T}$ under the assumption of a suitably controlled volume growth at infinity of geodesic balls.
	\begin{rmk}
	The condition on the volume growth is very mild on a Riemannian manifold without boundary and, for instance, is strictly implied by an appropriate corresponding conditions on the curvature of the manifold. In the case of a manifold with a nonempty boundary $\partial M$ it is in general not possible to obtain informations about the volume of geodesic balls from curvature hypoteses, indeed, as it is shown in \cite{ABB1} in general no curvature comparison theorems hold in this framework. Thus, the hypoteses on the volume growth seems to be more adequate in this case.
	\end{rmk}
We assume that $T$ satisfies
	\beqs
	0<T_-(r(x))\leq T(X,X)\leq T_+(r(x))
	\eeqs
for all $X\in T_xM$, $\left| X\right|=1$, $x\in\partial B_r(x)$, and some $T_{\pm}\in \mathrm{C}^0(\R_0^+)$. Furthermore, set
	\beqs
	\Theta(r(x))=\max_{[0,r(x)]}T_+(s)\,.
	\eeqs 

The following table defines our spaces of admissible functions.

\begin{center}
    \begin{tabular}{| l | l | m{5cm} | m{2,3cm} |}
    \hline
    Space & Regularity & Boundary behaviour & Test space\\ 
    \hline
    $\uB$ & $\mathrm{C}^{0}(M)\cap \mathrm{W}^{1,2}_{\mathrm{loc}}(M)$ & $\forall x\in\partial M,\, \exists\, \ep, \tau>0$ such that $\forall\,\, 0\leq\psi\in\mathrm{L}^{2}_{\mathrm{loc}}(M)$, $\int_{C_x(B_{\ep}(x),\tau)}\psi\, T(\nabla u,\nabla\rho)\geq 0$ & $\phi\in \mathrm{W}^{1,2}_{0}(M)$, \newline $\phi\geq 0$,\newline $\left.\phi\right|_{\partial M}\equiv 0$\\
    \hline
    $\uA$ & $\mathrm{C}^{0}(M)\cap \mathrm{W}^{2,2}_{\mathrm{loc}}(M)$ & $T(\nabla u,\nu)\leq 0$ on $\partial M$ & $\phi\in \mathrm{W}^{1,2}_{0}(M)$, \newline $\phi\geq 0$\\
    
    \hline
    \end{tabular}
\end{center}
\noindent
Where the $C_x(B_{\ep}(x),\tau)$ is the Fermi cylinder defined by \rf{fermi}. We also set the following.
	\begin{defi}
	For $K\subseteq\partial M$ and $u\in\uB$,
		\beqs
		H_u(K)=\inf_{x\in K}\left\{\tau(x)\,\,: \forall\,\, 0\leq\psi\in\mathrm{L}^{2}_{\mathrm{loc}}(M), \int_{C_x(B_{\ep}(x),\tau(x))}\psi\, T(\nabla u,\nabla\rho)\geq 0\right\}.
		\eeqs 
	Clearly, if $K$ is compact, then $H_u(K)>0$.
	\end{defi}
We are now ready to prove the next result. Although stated in different terms, that is, as sufficient condition for the validity of the $q$-$\partial$WMP, it is basically a generalization of Theorem A of \cite{PRSvol} to the case of manifolds with boundary. Thus its proof follows the lines of the argument used in the proof of the aforementioned Theorem A. However, due to the very subtle technicalities involved in the reasoning, we feel necessary, for a better understanding and for the ease of the reader, to provide a complete and detailed proof exposition in this new setting.
	\begin{thm}\label{suffvol}
	Let $\dM$ be a complete, noncompact, Riemannian manifold with boundary and denote with $r$ the distance function from a fixed point $o\in M$. Let $q\in \mathrm{C}^0(M)$, $q\geq 0$, and such that $q(x)\leq Q(r(x))$ where $Q(t)$ is positive, nondecreasing, satisfying
		\beq\label{q1}
		\Theta(t)Q(t)=o(t^{2})\quad\quad\hbox{as $t\rightarrow+\infty$}
		\eeq
	and
		\beq\label{q2}
		\liminf_{t\rightarrow+\infty}\frac{\Theta(t)Q(t)\log\vol B_t}{t^{2}}<+\infty\,.
		\eeq
	If $u\in\uB$ or $u\in\uA$ is such that $u^*=\sup_Mu<+\infty$ then it satisfies the $q$-$\partial WMP$ on $M$ for $L$.
	\end{thm}
	
	\begin{proof}
	Assume, by way of contradiction, that the space $\uB$ (respectively $\uA$) is not $L$-admissible for the $q$-$\partial WMP$ on $M$. We may suppose that, for some $\gamma<u^*$ and $u\in\uB$ (respectively $\uA$)  we have
	\beqs
	Lu\geq \frac{B}{Q(r(x))} \quad\hbox{on $\Omega_\gamma$}
	\eeqs
for some $B>0$ that, without loss of generality we can suppose to be $1$. Fix $0<\eta<1$. By choosing $\gamma$ sufficiently close to $u^*$, we may suppose that
\[
\Gamma=\gamma-u^*+\eta\geq\frac{\eta}{2}>0,
\]
so that, having defined $v=u-u^*+\eta$, we have
\[
v^*=\sup v=\eta, \quad \Omega^{v}_\Gamma=\Omega^{u}_\gamma,
\]
where $\Omega^{v}_\Gamma$ is defined as
	\[
	\Omega^{v}_\Gamma=\left\{x\in M\,:\,v(x)>\Gamma\right\}.
	\] 
Furthermore,
\beq
\label{VG13}
Lv\geq\frac{1}{Q(r(x))} \quad\hbox{on $\Omega^{v}_\Gamma$.}
\eeq
Choose $R_0>0$ large enough that $B_{R_0}\cap\Omega^{v}_\Gamma\neq\emptyset$. For a fixed $R\geq R_0$ let $\psi_R:M\rightarrow [0,1]$ be a smooth cut-off function such that
	\beq\label{VG14}
		\begin{array}{lll}
    	\textit{i)} & \psi_R\equiv 1 & \hbox{on $B_R$;} \\
    	\textit{ii)} & \psi_R\equiv 0 & \hbox{on $M\setminus B_{2R}$;} \\
    	\textit{iii)} & |\nabla\psi_R|\leq\frac{C_0}{R}\psi_R^{1/2}, & \hbox{}
		\end{array}
	\eeq
for some constant $C_0>0$. Note that requirement \textit{iii)} is possible because the exponent $1/2$ is less than $1$. Next, let $\lambda:\R{}\rightarrow\R{}^+_0$
be a $\mathcal{C}^1$ function such that
	\beq\label{VG15}
		\begin{array}{lll}
    	\textit{i)} & \lambda\equiv 0 & \hbox{on $(-\infty,\Gamma]$;} \\
    	\textit{ii)} & \lambda'(t)\geq 0 & \hbox{on $\R$;} \\
    	\textit{iii)} & \lambda\leq 1. & \hbox{}
  		\end{array}
	\eeq
Fix $\alpha>2$ and $0\leq\beta_R\in\mathrm{Lip}(B_{2R}\cap\Omega^v_\Gamma)$ to be determined later. Consider the function $\phi_R$ defined by
	\beq\label{VG16}
		\phi_R=\beta_R\psi_R^{2\alpha}\lambda(v)v^{\alpha-1} \quad \hbox{on $\Omega^v_\Gamma$}
	\eeq
and $\phi_R\equiv 0$ outside $\Omega_{\Gamma}^{v}$. Note that $\phi_R\equiv 0$ off $B_{2R}\cap\Omega^v_\Gamma$ and moreover $\phi_R\in\mathrm{W}^{1,2}_{0}(M)$. For future use it can be checked that the weak gradient of $\psi_R$ satisfies the following identity
	\beqs
	\begin{aligned}
	\nabla\phi_R & =\psi_R^{2\alpha}\lambda(v)v^{\alpha-1}\nabla\beta_R+2\alpha\beta_R\psi_R^{2\alpha-1}\lambda(v)v^{\alpha-1}\nabla\psi_R\\
	& \quad\quad+\beta_R\psi_R^{2\alpha}\lambda'(v)v^{\alpha-1}\nabla v+(\alpha-1)\beta_R\psi_R^{2\alpha}\lambda(v)v^{\alpha-2}\nabla v\,.\\
	\end{aligned}
	\eeqs
For the ease of notation we set
	\beqs
	T_v=\frac{T(\nabla v,\nabla v)}{|\nabla v|^2}\,,
	\eeqs
furthermore,
	\beqs
	|T(\nabla v,\nabla\psi_R)|\leq \sqrt{\frac{T(\nabla v,	\nabla v)}{|\nabla v|^2}}|\nabla v|
\sqrt{\frac{T(\nabla\psi_R,\nabla\psi_R)}{|\nabla\psi_R|^2}}|\nabla \psi_R|\leq
T_v^{1/2}T^{1/2}_{+}(R)|\nabla v||\nabla\psi_R|,
	\eeqs
that is,
	\beq\label{VG18}
	|T(\nabla v,\nabla\psi_R)|\leq T_v^{1/2}T^{1/2}_{+}(R)|\nabla v||\nabla\psi_R|\,.
	\eeq
Next, we consider two different cases.\\

Case I: $u\in\uB$.\\
In this case for $R\geq R_0$ we consider the function $0\leq\beta_R\in\mathrm{Lip}(B_{2R}\cap\Omega^v_\Gamma)$ defined by
	\beq\label{betar}
	\beta_R(x)=
	\left\{
	\begin{array}{ll}
	\frac{1}{\ep}\rho(x)&\hbox{on $M_{\ep}\cap B_{2R}\cap\Omega^v_\Gamma$}\\
	1&\hbox{on $\left(M\setminus M_{\ep}\right)\cap B_{2R}\cap\Omega^v_\Gamma$}
	\end{array}
	\right.
	\eeq
where 
	\beq\label{betaep}
	\ep=\ep(R)=\min\left\{\inj_{\rho}(\partial M\cap B_{2R}),\,H_u(\partial M\cap B_{2R})\right\}\,,
	\eeq
with 
	\beqs
	\inj_{\rho}(U)=\sup\left\{\tau\in\R^+\,\,:C_{x}(U,\tau)\cap\cut_{\rho}(\partial M)=\emptyset\right\},
	\eeqs
and $H_u(\partial M\cap B_{2R})$ as in Definition 3.4. Since $\partial M\cap B_{2R}\subset\subset\partial M$, it follows that $\ep(R)>0$ for $R>R_0$ (see for instance \cite{MM}), and $\beta_R$ is well defined.
We note that for $S\geq R$ we have the trivial inclusion $B_{2R}\subseteq B_{2S}$, thus, from \rf{betaep} it follows that $\ep(S)\leq\ep(R)$. In particular this implies that, for $S\geq R$, $0\leq\beta_S\in\mathrm{Lip}(B_{2R}\cap\Omega^v_\Gamma)$ and moreover
	\beq\label{betars}
	\beta_S\geq\beta_R\quad\hbox{on $B_{2R}$.}
	\eeq
With this choice of $\beta_R$ we have that $0\leq\phi_R\in \mathrm{W}^{1,2}_{0}(M)$ and $\left.\phi\right|_{\partial M}\equiv 0$. Thus $\phi_R$ is an admissible test function for $u\in\uB$. Recalling that $\lambda'\geq 0$ and using $\phi_R$ to test inequality (\ref{VG13}) we get
	\beqs
	\begin{aligned}
	0 & \geq\int_{B_{2R}}\psi_R^{2\alpha}\lambda(v)v^{\alpha-1}T(\nabla v,\nabla\beta_R)+2\alpha\beta_R\psi_R^{2\alpha-1}\lambda(v)v^{\alpha-1}T(\nabla v,\nabla\psi_R)\\
	& \quad\quad+\int_{B_{2R}}\beta_R\psi_R^{2\alpha}\lambda(v)v^{\alpha-1}\frac{1}{Q(r(x))}+(\alpha-1)\beta_R\psi_R^{2\alpha}\lambda(v)v^{\alpha-2}T_v\left|\nabla v\right|^2\,.\\
	\end{aligned}
	\eeqs
If we set
	\beqs
	I_R(\alpha)=\int_{B_{2R}}\psi_R^{2\alpha}\lambda(v)v^{\alpha-1}T(\nabla v,\nabla\beta_R)\,,
	\eeqs
then using \rf{VG18}  and rearranging, we obtain
	\beqs
	\begin{aligned}
	\int_{B_{2R}}\frac{\beta_R\psi_R^{2\alpha}\lambda(v)v^{\alpha-1}}{Q(r(x))} & \leq-I_R(\alpha)-(\alpha-1)\int_{B_{2R}}\beta_R\psi_R^{2\alpha}\lambda(v)v^{\alpha-2}T_v\left|\nabla v\right|^2\\
	& \quad\quad+2\alpha\int_{B_{2R}}\beta_R\psi_R^{2\alpha-1}\lambda(v)v^{\alpha-1}T^{1/2}_vT^{1/2}_+\left|\nabla v\right|\left|\nabla\psi_R\right|\,.\\
	\end{aligned}
	\eeqs
We apply to the second integral on the right hand side the inequality
	\beqs
	ab\leq\sigma\frac{a^2}{2}+\frac{b^2}{2\sigma}
	\eeqs
with 
	\beqs
	\begin{aligned}
	a & = \psi_R^{\alpha}v^{\alpha /2-1}T^{1/2}_v\left|\nabla v\right|\,,\\
	b & = \psi_R^{\alpha-1}v^{\alpha /2}T^{1/2}_+\left|\nabla\psi_R\right|\,,
	\end{aligned}
	\eeqs
and $\sigma=\frac{\alpha-1}{\alpha}$ so that the first integral on the right hand side cancels out. Indeed, we have
	\beq\label{VG21}
	\begin{aligned}
	\int_{B_{2R}}\frac{\beta_R\psi_R^{2\alpha}\lambda(v)v^{\alpha-1}}{Q(r(x))} & \leq-I_R(\alpha)+\frac{\alpha^2}{\alpha-1}\int_{B_{2R}}\beta_R\psi_R^{2\alpha-2}\lambda(v)v^{\alpha}T_+\left|\nabla\psi_R\right|^2\,.\\
	\end{aligned}
	\eeq
Now, in order to control the first term on the right hand side, we note that from the definition of $\beta_R$ it follows that	
	\beqs
	I_R(\alpha)=\frac{1}{\ep}\int_{M_{\ep}\cap B_{2R}\cap\Omega^v_{\Gamma}}\psi_R^{2\alpha}\lambda(v)v^{\alpha-1}T(\nabla v,\nabla\rho)\,,
	\eeqs
thus, since $v\in\uB$, $\psi_R^{2\alpha}\lambda(v)v^{\alpha-1}$ is locally bounded (indeed continuous), from the choice \rf{betaep} we conclude that
	\beq\label{Ira}
	I_R(\alpha)\geq 0\,,
	\eeq
for $R\geq R_0$.\\
Now, since $Q$ is non-decreasing, $Q(r(x))\leq Q(2R)$ on the support of $\psi$ and the left hand side of \rf{VG21} is bounded from below by
	\beq\label{VG23}
	\frac{1}{Q(2R)}\int_{B_{2R}}\beta_R\psi_R^{2\alpha}\lambda(v)v^{\alpha-1}\,.
	\eeq
On the other hand
	\beqs
	\frac{\alpha}{\alpha-1}\leq 2 \quad\hbox{ for $\alpha\geq 2$ ,}
	\eeqs
and furthermore, using \rf{VG14} \emph{iii)}, we may write
	\beqs
	\psi_R^{2\alpha-2}|\nabla\psi_R|^{2}=\psi_R^{2\alpha-1}(\psi_R^{-1/2}|\nabla\psi_R|)^{2}\leq
\psi_R^{2\alpha-1}\frac{C_0^2}{R^{2}}.
	\eeqs
Finally, we recall that 
	\beqs
	T_{+}(r(x))\leq \Theta(2R)\quad \hbox{ on $B_{2R}$.}
	\eeqs
Thus, the right hand side of \rf{VG21} can be estimated from above by
	\beqs
	2\alpha\,\Theta(2R)\frac{C_0^2}{R^{2}}\int \beta_R\psi_R^{2\alpha-1}\lambda(v)v^{\alpha} \,.
	\eeqs

Now, we apply H\"older's inequality with conjugate exponents $\alpha/(\alpha-1)$ and $\alpha$
to estimate from above this last expression with
	\beq\label{VG24}
	2\alpha\,\Theta(2R)\frac{C_0^2}{R^{2}}
\left(\int \beta_R\psi_R^{2\alpha}\lambda(v)v^{\alpha-1}\right)^{\frac{\alpha-1}{\alpha}}
\left(\int \beta_R\psi_R^{\alpha}\lambda(v)v^{2\alpha-1}\right)^{\frac{1}{\alpha}}.
	\eeq
Using \rf{Ira}, \rf{VG23}, and \rf{VG24} into \rf{VG21}, after a rearrangement we have
	\beqs
	\int \beta_R\psi_R^{2\alpha}\lambda(v)v^{\alpha-1}\leq
\left(2\alpha\,\Theta(2R)Q(2R)\frac{C_0^2}{R^{2}}\right)^\alpha
\int \beta_R\psi_R^{\alpha}\lambda(v)v^{2\alpha-1}.
	\eeqs
Recalling that $\psi_R\equiv 1$ on $B_R$, $\psi_R\equiv 0$ on $M\setminus B_{2R}$ and that $\eta/2\leq v\leq \eta$ on $\Omega^v_\Gamma$ when $\lambda(v)>0$,
we deduce that
	\beqs
	\int_{B_R}\beta_R\lambda(v)\leq\left(\eta\alpha\,2^{(2\alpha-1)/\alpha}\Theta(2R)Q(2R)\frac{C_0^2}{R^{2}}\right)^\alpha
\int_{B_{2R}}\beta_R\lambda(v).
	\eeqs
Moreover, using \rf{betars} with $S=2R$, we get
	\beq\label{VG25}
	\begin{aligned}
	\int_{B_R}\beta_R\lambda(v) & \leq\frac{1}{2}\left(\eta\alpha\,\Theta(2R)Q(2R)\frac{C_1}{R^{2}}\right)^\alpha
\int_{B_{2R}}\beta_{2R}\lambda(v)\\
& \leq \left(\eta\alpha\,\Theta(2R)Q(2R)\frac{C_1}{R^{2}}\right)^\alpha
\int_{B_{2R}}\beta_{2R}\lambda(v).
	\end{aligned}
	\eeq
with
	\beqs
	C_1=4C_0^2
	\eeqs
We now set
	\beqs
	\alpha=\alpha(R)=\frac{1}{2\eta C_1}\frac{R^{2}}{\Theta(2R)Q(2R)}
	\eeqs
(which, as follows from \rf{q1}, is $\geq 2$ for $R$ sufficiently large) so that we can rewrite \rf{VG25} as
	\beq\label{VG26}
	\int_{B_R}\beta_R\lambda(v)\leq\left(\frac{1}{2}\right)^{\frac{1}{2\eta C_1}\frac{R^{2}}{\Theta(2R)Q(2R)}}\int_{B_{2R}}\beta_{2R}\lambda(v),
	\eeq
for each $R\geq R_0$. Note that $C_1$ is independent of $R_0$ and $\eta$. We now need the following result proved in \cite{PRSvol} (see Lemma 1.1).
\begin{lem}
\label{VG27}
Let $G,F:[R_0,+\infty)\fle\R{}^+_0$ be non-decreasing functions such that for some constants $0<\Lambda<1$ and $B,\theta>0$
\beq
\label{VG28}
G(R)\leq \Lambda^{B\frac{R^\theta}{F(2R)}}G(2R), \text{ for each } R\geq R_0.
\eeq
Then there exists a constant $S=S(\theta)>0$ such that for each $R\geq 2R_0$
\beq
\label{VG29}
\frac{F(R)}{R^\theta}\log G(R)\geq \frac{F(R)}{R^\theta}\log G(R_0)+SB\log(\frac{1}{\Lambda}).
\eeq
\end{lem}
We set $G(R)=\int_{B_R}\beta_R\lambda(v)$. $G$ is non-decreasing, indeed, using the monotonicity of integral and \rf{betars}, for $S\geq R$ 
	\beqs
	G(S)=\int_{B_S}\beta_S\lambda(v)\geq\int_{B_R}\beta_S\lambda(v)\geq\int_{B_R}\beta_R\lambda(v)=G(R)\,.
	\eeqs
Thus we can apply Lemma\rl{VG27} with $G(R)$ as above, $\theta=2$, $\Lambda=1/2$, $B=\frac{1}{2\eta  C_1}$, $F(R)=Q(R)\Theta(R)$ to deduce that for each $R\geq 2R_0$
	\beq\label{VG31}
	\frac{Q(R)\Theta(R)}{R^{2}}\log \int_{B_R}\beta_R\lambda(v)\geq
\frac{Q(R)\Theta(R)}{r^{2}}\log \int_{B_R}\beta_R\lambda(v)+
\frac{1}{24\eta\, C_1}\log 2.
	\eeq
Now since $\sup \beta_R=\sup\lambda=1$, letting $R\fle+\infty$ in \rf{VG31} and using \rf{q1} we obtain
	\beqs
	\begin{aligned}
	\liminf_{R\fle+\infty}\frac{Q(R)\Theta(R)}{R^{2}}\log \vol B_R & \geq\liminf_{R\fle+\infty}\frac{Q(R)\Theta(R)}{R^{2}}\log \int_{B_R}\beta_R\lambda(v)\\
	& \geq \frac{1}{24\eta\, C_1}\log 2\,,
	\end{aligned}
	\eeqs
with $C_1$ independent of $\eta$. Letting $\eta\fle 0^+$ we contradict \rf{q2}. This completes the proof of the theorem.\\

Case II: $u\in\uA$.\\
In this case the proof is simpler, indeed we take $\beta_R\equiv 1$ for each $R$, then the boundary behaviour of $\uA$ permits to estimate immediately the boundary term (the $I_R(\alpha)$ term of the previous case), obtaining inequality \rf{VG25}.\\ 
Then the proof follows that of Case I.
\end{proof}

From the theorem above we deduce easily the following result which extends Theorem A of \cite{PRSvol}.

	\begin{thm}\label{thma} 
	Let $\dM$ be a Riemannian manifold with boundary. Let $f\in \mathrm{C}^0(\R)$ and assume that $u\in\uB$ (or $\uA$) satisfy $u^* = sup_M u < +\infty$, and
		\beq\label{Lbf}
		L u \geq b(x)f(u)  \quad\hbox{on $\Omega_{\gamma}$}
		\eeq
	where as usual
		\beqs
		\Omega_{\gamma} = \left\{x \in M : u(x) > \gamma\right\}\, ,
		\eeqs
	for some $\gamma < u^*$, $b(x)$ is a continuous positive function on $M$ satisfying
		\beq\label{b1}
		b(x) \geq\frac{1}{Q(r(x))} \quad \hbox{outside a compact set}
		\eeq
	and $Q(t)$ is as Theorem \ref{suffvol}. 
	If $Q$ satisfies (\ref{q1}) and (\ref{q2}) then $f(u^*) \leq 0$.
	\end{thm}
		\begin{proof}
		Assume, by way of contradiction, that $f(u^*)=2\ep>0$, then by the continuity of $f$ and $u$, there exists a $\gamma<\gamma_{\ep}<u^*$ such that
			\beqs
			f(u)>\ep\quad\quad\hbox{on $\Omega_{\gamma_{\ep}}$},
			\eeqs
		thus, from \rf{Lbf} it follows that
			\beqs
			\inf_{\Omega_{\gamma_{\ep}}}\frac{1}{b(x)}Lu\geq\inf_{\Omega_{\gamma_{\ep}}}f(u)>\ep>0\,,
			\eeqs
		which is impossible, since by Theorem \ref{suffvol} $u$ satisfies the $\frac{1}{b}$-$\partial WMP$ on $M$.
		\end{proof}
The following \emph{a priori} estimate (in fact its consequence Corollary \ref{cor3.9} below) extends Theorem B of \cite{PRSvol} to the case of manifolds with boundary; it will be crucial in the proofs of Theorem A and Corollary B. Analogously to Theorem \ref{suffvol}, the proof of the result follows the lines of the aforementioned Theorem B of \cite{PRSvol} but we feel necessary to provide a complete and detailed proof for the ease of the reader.
	\begin{thm}\label{thmb}
	Let $\dM$ be a Riemannian manifold with boundary. Let $b$, $Q$, $T$, and $\Theta$ be as above. Assume that $u\in\uB$ (or $\uA$) satisfies
		\beq\label{VG35}
		L u\geq b(x)f(u) \quad \hbox{on $\Omega_{\gamma}$}
		\eeq
	for some $\gamma < u^*\leq+\infty$, where $f$ is a continous function on $\R$ such that
		\beq\label{VG36}
		\liminf_{t\rightarrow+\infty}\frac{f(t)}{t^{\sigma}}>0
		\eeq 
	for some $\sigma > 1$. If (\ref{q1}) and (\ref{q2}) hold true, then $u$ is bounded above.
	\end{thm}
	\begin{proof}
	Assume, by way of contradiction, that $u$ is not bounded above, so that the set
		\beqs
		\Omega_\gamma=\{ x\in M : u(x)>\gamma \}
		\eeqs
	is nonempty for each $\gamma>0$. By increasing $\gamma$, if necessary, we may assume that $f(t)\geq Bt^\sigma$ if $t\geq\gamma$. For the ease of notation, we let $B=1$ so that on $\Omega_\gamma$
		\beq\label{VGlu}
		\div\left(\widetilde{T}(\nabla u)\right)\geq b(x)u^\sigma,
		\eeq
	weakly.

	Clearly we may also assume that $b(x)$ is bounded above. Let $R_0>0$ be large enough that $\Omega_\gamma\cap B_{R_0}\neq\emptyset$. Now we will proceed as in the proof of Theorem\rl{suffvol}, that is, we are going to define a suitable family of test functions in order to get a contradiction. Fix $\xi>1$ satisfying
		\beq\label{VG37}
		1-\frac{2}{\sigma-1}\left(1-\frac{1}{\xi}\right)>0
		\eeq
	For each $R\geq R_0$ let $\psi=\psi_R:M\rightarrow [0,1]$ be a smooth cut-off function such that
	\beq\label{VG38}
		\begin{array}{lll}
    	\textit{i)} & \psi_R\equiv 1 & \hbox{on $B_R$;} \\
    	\textit{ii)} & \psi_R\equiv 0 & \hbox{on $M\setminus B_{2R}$;} \\
    	\textit{iii)} & |\nabla\psi_R|\leq\frac{C_0}{R}\psi_R^{1/{\xi}}, & \hbox{}
		\end{array}
	\eeq
for some constant $C_0>0$. Note that this latter requirement \textit{iii)} is possible since $\xi>1$. Next, let $\lambda:\R{}\rightarrow\R{}^+_0$ be a $\mathrm{C}^1$ function such that
	\beqs
		\begin{array}{lll}
    	\textit{i)} & \lambda\equiv 0 & \hbox{on $(-\infty,\gamma]$;} \\
    	\textit{ii)} & \lambda'(t)\geq 0 & \hbox{on $\R$;} \\
    	\textit{iii)} & \sup\lambda= \frac{1}{\sup_Mb}>0. & \hbox{}
  		\end{array}
	\eeqs	
	Finally, fix $\alpha>2\sigma$, $\mu>0$, and $0\leq\beta_R\in\mathrm{Lip}(B_{2R}\cap\Omega^v_\Gamma)$ to be determined later. Consider the function $\phi_R$ defined by
	\beq\label{VG16}
		\phi_R=\beta_R\psi_R^{\alpha}\lambda(u)u^{\mu} \quad \hbox{on $\Omega_{\gamma}$}
	\eeq
and $\phi_R\equiv 0$ outside $\Omega_{\gamma}$. Note that $\phi_R\equiv 0$ off $B_{2R}\cap\Omega_{\gamma}$ and moreover $\phi_R\in\mathrm{W}^{1,2}_{0}(M)$. It can be checked that the weak gradient of $\phi_R$ satisfies
	\beqs
	\begin{aligned}
	\nabla\phi_R & =\psi_R^{\alpha}\lambda(u)u^{\mu}\nabla\beta_R+\alpha\beta_R\psi_R^{\alpha-1}\lambda(u)u^{\mu}\nabla\psi_R\\
	& \quad\quad+\beta_R\psi_R^{\alpha}\lambda'(u)u^{\mu-1}\nabla u+\mu\beta_R\psi_R^{\alpha}\lambda(u)u^{\mu-1}\nabla u\,.\\
	\end{aligned}
	\eeqs
	Now we proceed as in the proof of Theorem\rl{suffvol} using the function $\phi_R$ to test the inequality \rf{VGlu}. We recall that $\lambda'>0$, use \rf{VG18}, and  furthermore choose $\beta_R$ according to the function space of $u$ as above, in order to get rid of the boundary term. Thus we obtain
	\beqs
	\begin{aligned}
	\int_{B_{2R}}\beta_R\psi_R^{\alpha}\lambda(u)u^{\mu+\sigma}b(x) & \leq-\mu\int_{B_{2R}}\beta_R\psi_R^{\alpha}\lambda(u)u^{\mu-1}T_u\left|\nabla u\right|^2\\
	& \quad\quad+\alpha\int_{B_{2R}}\beta_R\psi_R^{\alpha-1}\lambda(u)u^{\mu}T^{1/2}_uT^{1/2}_+\left|\nabla u\right|\left|\nabla\psi_R\right|\,.\\
	\end{aligned}
	\eeqs
We apply to the second integral on the right hand side the inequality
	\beqs
	ab\leq\ep\frac{a^2}{2}+\frac{b^2}{2\ep}
	\eeqs
with 
	\beqs
	\begin{aligned}
	a & = \psi_R^{\alpha/2}u^{(\mu-1)/2}T^{1/2}_u\left|\nabla u\right|\,,\\
	b & = \psi_R^{\alpha/2-1}u^{(\mu+1)/2}T^{1/2}_+\left|\nabla\psi_R\right|\,,
	\end{aligned}
	\eeqs
and $\ep=\frac{2\mu}{\alpha}$ so that the first integral on the right hand side cancels out and we obtain
	\beq\label{VG39}
	\int_{B_{2R}}\beta_R\psi_R^{\alpha}\lambda(u)u^{\mu+\sigma}b(x) \leq\frac{\alpha^2}{4\mu}\int_{B_{2R}}\beta_R\psi_R^{\alpha-2}\lambda(u)u^{\mu+1}T_+\left|\nabla\psi_R\right|^2\,.
	\eeq
	Multiplying and dividing by $b(x)^{1/p}$ in the integral on the right hand side, and applying H\"older's inequality with conjugate exponents $p$ and $q$, yields
\begin{eqnarray*}
\int\beta_R\psi_R^{\alpha-2}\lambda(u)u^{\mu+1}T_{+}|\nabla\psi|^2\leq
\left(\int\beta_R\psi_R^\alpha b(x)\lambda(u)u^{p(\mu+1)}\right)^{1/p}\\
\times \left(\int\beta_R\psi_R^{\alpha-2q(1-1/\xi)}\lambda(u)b(x)^{1-q}T_{+}^{q}\left(\frac{|\nabla\psi_R|}{\psi_R^{1/\xi}}\right)^{2q}\right)^{1/q},
\end{eqnarray*}
provided
\beq
\label{VG40}
\alpha-2q(1-1/\xi)>0.
\eeq
Choosing $\displaystyle{p=\frac{\mu+\sigma}{\mu+1}}>1$ since $\sigma>1$, the first integral on the right hand side of the above inequality is equal to the integral on the left hand side of \rf{VG39}. Thus, inserting into this latter and simplifying, we obtain
\beqs
\int\beta_R\psi_R^\alpha b(x)\lambda(u)u^{\mu+\sigma}\leq\left(\frac{\alpha^2}{4\mu}\right)^q\int\beta_R\psi_R^{\alpha-2q(1-1/\xi)}\lambda(u)b(x)^{1-q}T_{+}^{q}\left(\frac{|\nabla\psi_R|}{\psi_R^{1/\xi}}\right)^{2q}.
\eeqs
Since $u>\gamma$ on $\Omega_\gamma$ and $\psi\equiv 1$ on $B_R$,
\[
\gamma^{\mu+\sigma}\int_{B_R}\beta_R b(x)\lambda(u)\leq\int\beta_R\psi_R^\alpha b(x)\lambda(u)u^{\mu+\sigma}.
\]
On the other hand, using \rf{VG38} ii), iii), the fact that $\psi_R$ is supported on $B_{2R}$, and the monotonicity of $\beta_S$ with respect to $S$, we have
\beqs
\begin{aligned}
\left(\frac{\alpha^2}{4\mu}\right)^q\int\beta_R\psi_R^{\alpha-2q(1-1/\xi)}\lambda(u)b(x)^{1-q}T_{+}^{q}\left(\frac{|\nabla\psi_R|}{\psi_R^{1/\xi}}\right)^{2q}\\
\leq
\left(\frac{\alpha^2}{4\mu}\frac{C_0^2}{R^2}\sup_{B_{2R}}\frac{T_{+}}{b(x)}\right)^q
\int_{B_{2R}}\beta_{2R}b(x)\lambda(u).
\end{aligned}
\eeqs
We use these two latter inequalities, the fact that $b(x)\geq Q(r(x))^{-1}$ with $q$ non-decreasing, the validity of 
\beq
\label{VG42}
T_{+}(r(x))\leq \Theta(2R)
\eeq
on $B_{2R}$, and
\[
q=\frac{\mu+\sigma}{\sigma-1}
\]
to obtain
\beq
\label{VG43}
\int_{B_R}\beta_R b(x)\lambda(u)\leq\left(\frac{C_0^2}{4\gamma^{\sigma-1}}\frac{\Theta(2R)Q(2R)}{R^{2}}\left(\frac{\alpha}{\mu}\right)\alpha\right)^{\frac{\mu+\sigma}{\sigma-1}}\int_{B_{2R}}\beta_{2R}b(x)\lambda(u)\,.
\eeq
Now we choose
\[
\alpha=\mu+\sigma=\frac{1}{C_0^2}\gamma^{\sigma-1}\frac{R^{2}}{\Theta(2R)Q(2R)}
\]
so that \rf{VG37} implies that \rf{VG40} holds. Moreover, because of \rf{q1}, $\alpha\rightarrow+\infty$ as $R\rightarrow+\infty$. Hence, for $R$ sufficiently large $\frac{\alpha}{\mu}\leq 2$. It follows that, for such values of $R$, \rf{VG43} gives
\beq
\label{VG44}
\int_{B_R}\beta_Rb(x)\lambda(u)\leq\left(\frac{1}{2}\right)^{\frac{\gamma^{\sigma-1}}{C_0^2(\sigma-1)}\frac{R^{2}}{\Theta(2R)Q(2R)}}
\int_{B_{2R}}\beta_{2R}b(x)\lambda(u).
\eeq

We let
\[
G(R)=\int_{B_R}\beta_Rb(x)\lambda(u)
\]
and
\[
F(R)=\Theta(R)Q(R)
\]
be defined on $[R_0,+\infty)$ for some $R_0$ sufficiently large such that \rf{VG44} holds for $R\geq R_0$. Then
\[
G(R)\leq \left(\frac{1}{2}\right)^{B\frac{R^{2}}{F(2R)}}G(2R)
\]
with $B=\frac{\gamma^{\sigma-1}}{C_0^2(\sigma-1)}>0$. Then by Lemma\rl{VG27}, there exists a constant $S>0$ such that, for each $R\geq 2R_0$
\[
\frac{Q(R)\Theta(R)}{R^{2}}\log\int_{B_R}\beta_Rb(x)\lambda(u)\geq
\frac{Q(R)\Theta(R)}{R^{2}}\log\int_{B_R}\beta_Rb(x)\lambda(u)+SB\log 2,
\]
To reach the desired contradiction, we recall that $\sup\lambda=\frac{1}{\sup_Mb}>0$ so that $b(x)\lambda(u)\leq 1$. Taking $R$ going to $+\infty$ in the above and using \rf{q1} we deduce
\[
\liminf_{R\rightarrow+\infty}\frac{Q(R)\Theta(R)}{R^{2}}\log\vol{B_R}\geq SB\log 2=\frac{\gamma^{\sigma-1}}{C_0^2(\sigma-1)}S\log 2.
\]
This contradicts \rf{q2} by choosing $\gamma$ sufficiently large.
\end{proof}
As a consequence of Theorem \ref{thmb} we have the following \emph{a priori} estimate for solutions of the differential inequality \rf{Luba} below. The importance of this type of results can be hardly overestimated in PDE's Theory and it will be used in the next section.
	\begin{cor}\label{cor3.9}
	Let $\dM$ be a Riemannian manifold with boundary. Let $a(x)$, $b(x)\in \mathrm{C}^0(M)$ where $a(x) = a_+(x) -a_-(x)$, with $a_+$, $a_-$ respectively the positive and negative parts of $a$. Suppose that $\left\|a_-\right\|_{\infty}< +\infty$ and that $b(x)>0$ on $M$ satisfies (\ref{b1}). Assume furthermore that, for some $H >0$,
		\beqs
		\frac{a_-(x)}{b(x)}\leq H \quad\hbox{on $M$}. 
		\eeqs
	Let $u\in\uB$ (or $\uA$) be a non-negative solution of
		\beq\label{Luba}
		L u\geq b(x)u^{\sigma} + a(x)u \quad \hbox{on $\Omega_{\gamma}$}
		\eeq
	for some $\gamma < u^*\leq+\infty$, and	for some $\sigma > 1$.\\ 
	If $Q$ satisfies (\ref{q1}) and (\ref{q2}), then $u$ satisfies
	\beqs
		u(x) \leq H^{1\slash(\sigma-1)} \quad\hbox{on $\Omega_{\gamma}$}.
	\eeqs
	\end{cor}
		\begin{proof}
		The assumptions on $a(x)$ and $b(x)$ imply that 
			\beqs
			Lu\geq b(x)\left(u^{\sigma}-Hu\right)\quad\quad\quad\hbox{on $\Omega_{\gamma}$}\,,
			\eeqs
		thus, since
			\beqs
			\liminf_{t\rightarrow+\infty}\frac{t^{\sigma}-Ht}{t^{\sigma}}=1\,,
			\eeqs 
		it follows from Theorem \ref{thmb} that $u$ is bounded above. Furthermore, by Theorem \ref{thma} it follows that $\left(u^{*}\right)^{\sigma}-Hu^{*}\leq 0$ on $\Omega_{\gamma}$, which implies that
			\beqs
			u(x) \leq H^{1\slash(\sigma-1)} \quad\hbox{on $\Omega_{\gamma}$}.
			\eeqs
		\end{proof}
		
\section{Proof of the main results and other geometric applications}\label{secconf}
We apply the results of the previous section to prove our main theorems. We start by noting that on a smooth Riemannian manifold with smooth boundary the scalar curvature $s$ and the mean curvature of the boundary $h$ are smooth functions, namely $s\in\mathrm{C}^{\infty}(M)$ and $h\in\mathrm{C}^{\infty}(\partial M)$. Thus, by standard elliptic regularity theory, solutions $u$ of \rf{bYa} are smooth, indeed $u\in\mathrm{C}^{\infty}(M)$.
		\begin{proof}[Proof of Theorem A]
		From (\ref{bYa}) and (\ref{hsc}) we have that $u$ satisfies
			\beqs
			\begin{dcases}
			\Delta u-c_m\left(s(x)u-\widetilde{s}(x)u^{\frac{m+2}{m-2}}\right)=0 & \hbox{on $\Omega_{\gamma}$}\\
			\partial_{\nu}u\leq 0 & \hbox{on $\partial \Omega_{\gamma}\cup\left(\overline{\Omega}_{\gamma}\cap\partial M\right)$}\,.
			\end{dcases}
			\eeqs
		Now we apply Theorem \ref{thmb} to conclude the proof.
		\end{proof}
		\begin{proof}[Proof of Corollary B]
		First note that for $\widetilde{\g{\,}{\,}}=f^*\g{\,}{\,}=u^{\frac{4}{m-2}}\g{\,}{\,}$ the $\partial$-rigidity assumption on $f$ implies $\widetilde{h}(x)= u^{-\frac{2}{m-2}}h(x)$ on $\partial M$ so that \rf{conds} and $\widetilde{s}(x)=s(x)$ imply that the assumptions of Corollary \ref{cor3.9} are satisfied. Hence $u\leq 1$.\\ 
		We need to prove $u\geq1$. Toward this aim we observe that for the inverse diffeomorphism $\left(f^{-1}\right)^*\g{\,}{\,}=w^{\frac{4}{m-2}}\g{\,}{\,}$ with $w(y)=\frac{1}{u(f^{-1}(y))}$, $w$ satisfies
			\beqs
			\begin{dcases}
			\Delta w-c_ms(y)\left(w-w^{\frac{m+2}{m-2}}\right)=0 & \hbox{on $M$}\\
			\partial_{\nu}w+d_m\left(\widetilde{h}(y)w-h(y)w^{\frac{m}{m-2}}\right)=0 & \hbox{on $\partial M$.} 
			\end{dcases}
			\eeqs
		The result then follows from Corollary \ref{cor3.9} if we show that $\partial_{\nu}w=0$ on $\partial M$. Toward this aim we compute 
			\beq\label{dnuw}
			\begin{aligned}
			\partial_{\nu}w(y) & =-\frac{d_y\left(u\circ f^{-1}\right)[\nu_y]}{\left(u\circ f^{-1}\right)^2(y)}\\
			& = -\frac{\left(d_{f^{-1}(y)}u\right)[(f^{-1})_*\nu_y]}{\left(u\circ f^{-1}\right)^2(y)}\\
			\end{aligned}
			\eeq
		where $(f^{-1})_*\nu_y\in T_{f^{-1}(y)}M$ (see Chapter 3 of \cite{Lee} for the definition of the tangent space at points $x\in\partial M$), and since $f^{-1}$ is a conformal diffeomorphism it preserves the normal vectors at boundary, that is $(f^{-1})_*\nu_y=\mu(y)\nu_{f^{-1}(y)}$ for some positive function $\mu$. Set $x=f^{-1}(y)$, then from \rf{dnuw} and $\partial_{\nu}u=0$ 
			\beqs
			\begin{aligned}
			\partial_{\nu}w(f(x)) & = -\mu(f(x))\frac{d_{x}u[\nu_x]}{u^2(x)}\\
			& = -\mu(f(x))\frac{\partial_{\nu}u(x)}{u^2(x)}\\
			& = 0.
			\end{aligned}
			\eeqs
		Now, reasoning as above we conclude that $w\leq 1$, and therefore $u\geq 1$.
		\end{proof}
	\begin{rmk}\label{rigrmk}
	From (\ref{bYa}) it follows immediately that, for a conformal diffeomorphism, the condition of being $\partial$-rigid is equivalent to requiring
		\beq\label{rigeq}
		\widetilde{h}(x)=u^{-\frac{2}{m-2}}h(x)\quad\hbox{on $\partial M$}\,.
		\eeq
	From this equation we observe that a $\partial$-rigid diffeomorphism preserves pointwise the sign of the mean curvature.\\
	We observe that condition (\ref{rigeq}) is automatically satisfied whenever the boundary $\partial M$ is minimal with respect to the metric $\g{\,}{\,}$ and we look for diffeomorphisms preserving this property, that is, minimality of the boundary in the conformally deformed metric. Furthermore we have that if the mean curvatures $h$ and $\widetilde{h}$ have the same sign and do not vanish on $\partial M$, then the diffeomorphism is $\partial$-rigid if and only if $u$ is a solution of the overdetermined problem
		\beqs
		\begin{dcases}
		\Delta u-c_m\left(s(x)u-\widetilde{s}(x)u^{\frac{m+2}{m-2}}\right)=0 & \hbox{on $M$}\\
		u=\left(\frac{h(x)}{\widetilde{h}(x)}\right)^{\frac{m-2}{2}} & \hbox{on $\partial M$}\\
		\partial_{\nu}u=0 & \hbox{on $\partial M$}\,.
		\end{dcases}
		\eeqs
	In particular the conformal factor of a conformal diffeomorphism such that $\widetilde{s}=s$ and $\widetilde{h}=h$ on $\partial M$ is $\partial$-rigid if and only if it is a solution of the problem
		\beq\label{hover}
		\begin{dcases}
		\Delta u-c_ms(x)\left(u-u^{\frac{m+2}{m-2}}\right)=0 & \hbox{on $M$}\\
		u=1 & \hbox{on $\partial M$}\\
		\partial_{\nu}u=0 & \hbox{on $\partial M$}\,.
		\end{dcases}
		\eeq
	\end{rmk}
Other sufficient conditions for the $\partial$-rigidity of a conformal diffeomorphism can be deduced by imposing some restrictions on higher order extrinsic curvatures.\\ 
Toward this aim we recall some definitions. Let $\varphi:\Sigma^{m-1}\rightarrow M^{m}$ denote an immersion of a connected, $(m-1)$-dimensional Riemannian manifold and assume that it is oriented by a globally defined unit normal vector field $N$.\\
Let $A$ denote the second fundamental form of the immersion in the direction of $N$. Then, the $k$-mean curvatures of the hypersurface are defined by
	\[
	H_k= {m \choose k}^{-1}S_k,
	\]
where $S_0=1$ and, for $k=1,\ldots,m$, $S_k$ is the $k$-th elementary symmetric function of the principal curvatures of the hypersurface. In particular, $H_1=h$ is the mean curvature and $H_m$ is the Gauss-Kronecker curvature of $\Sigma$. In the case of a Riemannian manifold with boundary we can consider the $k$-mean curvatures of the immersion $\varphi:\partial M\rightarrow M$.\\ 
In the following discussion we modify our previous notation for the ease of the reader. Let $\dM$ be a Riemannian manifold of dimension $m$ with boundary and, for a smooth function $f$ on $M$, consider the pointwise conformal change of metric $\widetilde{\g{}{}}=\e^{2f}\g{}{}$. In the previous notation it was $\e^f=u^{\frac{2}{m-2}}$ for a positive smooth function. We know from equation (1.3) of \cite{Es}, that under the conformal transformation above, the second fundamental form of the boundary changes in the following way 
	\beqs
	\widetilde{A}_{ij}=e^f\left(A_{ij}+\partial_{\nu}fg_{ij}\right)
	\eeqs
where $g_{ij}$ are the components of the metric tensor $\g{}{}$. We note also that the components of the inverse of the metric tensor change as
	\beqs
	\widetilde{g}^{ij}=\e^{-2f}g^{ij}\,.
	\eeqs
The following lemma is well known (see for instance \cite{Ab}).
	\begin{lem}
	Let $\dM$ be a Riemannian manifold of dimension $m\geq 3$ with boundary. On $\partial M$ define
		\beq\label{deflam}
		\Lambda=m^2(m-1)\left(H_2-H_1^2\right)
		\eeq
	where $H_2$ and $H_1=h$ are the second and first mean curvatures of $\partial M$.
	Then, under the pointwise conformal change of metric $\widetilde{\g{}{}}=\e^{2f}\g{}{}$ with the obvious meaning of the notation we have
		\beq\label{lambda}
		\widetilde{\Lambda}=\e^{-2f}\Lambda \,.
		\eeq
	\end{lem}
We note that the quantity $\Lambda$ is the conformal Willmore integrand for surfaces immersed in $3$-manifolds, indeed its integral is a conformal invariant for immersed surfaces.\\
Next we exploit the formal similarity between equations (\ref{rigeq}) and (\ref{lambda}) to find sufficient conditions for a conformal deformation to be $\partial$-rigid. In particular the following result gives an explicit characterization of $\partial$-rigidity (see also its relation with the discussion before the statement of Corollary B).
	\begin{cor}
	Let $\dM$ be a complete manifold with boundary, dimension $m\geq 3$ and scalar curvature $s(x)$ such that (\ref{conds}) and (\ref{volQ}) hold. Then, any conformal diffeomorphism of $\dM$ into itself which preserves the scalar curvature, the sign of the mean curvature, and such that $\widetilde{H}_2=H_2\equiv 0$, is an isometry.
	\end{cor}
		\begin{proof}
		The idea is to show that any conformal transformation preserving the sign of the mean curvature and such that $\widetilde{H}_2=H_2\equiv 0$ is indeed $\partial$-rigid, so that we can apply Corollary B.\\
		From equations (\ref{deflam}) and (\ref{lambda})
			\beqs
			\left(\widetilde{H}_2-\widetilde{h}^2\right)=u^{-\frac{4}{m-2}}\left(H_2-h^2\right) \quad\hbox{on $\partial M$}\,,
			\eeqs
		now, since $\widetilde{H}_2=H_2\equiv 0$ and $h(x)$ ha the sign of $\widetilde{h}(x)$, it follows that
			\beqs
			\widetilde{h}=u^{-\frac{2}{m-2}}h
			\eeqs
		that is, the transformation is $\partial$-rigid.
		\end{proof}
We conclude the section with our last geometric result.
	\begin{proof}[Proof of Theorem C]
		The case $\widetilde{h}=h\equiv 0$ on $\partial M$ follows from Corollary B and Remark \ref{rigrmk}.\\
		In the general case assume, by way of contradiction, that $1<u^*\leq+\infty$, choosing $1<\gamma<u^*$ we have
			\beqs
			\begin{dcases}
			\Delta u=c_ms(x)\left(u-u^{\frac{m+2}{m-2}}\right) & \hbox{on $\Omega_{\gamma}$}\\
			\partial_{\nu}u\leq 0 & \hbox{on $\partial \Omega_{\gamma}$}\,\\
			\partial_{\nu}u=d_mh(x)\left(u^{\frac{2}{m-2}}-1\right)u & \hbox{on $\overline{\Omega}_{\gamma}\cap\partial M$}\,.
			\end{dcases}
			\eeqs
		Since $\gamma>1$, and $h\leq 0$ we deduce that
			\beqs
			\begin{dcases}
			\Delta u=c_ms(x)\left(1-u^{\frac{4}{m-2}}\right)u & \hbox{on $\Omega_{\gamma}$}\\
			\partial_{\nu}u\leq 0 & \hbox{on $\partial \Omega_{\gamma}\cup\left(\overline{\Omega}_{\gamma}\cap\partial M\right)$}\,
			\end{dcases}
			\eeqs
		Theorem \ref{thmb} implies that $u\leq 1$ on $\Omega_{\gamma}$, contradicting the assumption that $u^*>1$. This shows that $u\leq 1$ on $M$. 
		To conclude the proof we recall that the conformal factor of the inverse deformation $f^{-1}$ is $w(y)=\frac{1}{u(f^{-1}(y))}$ which satisfies
			\beqs
			\begin{dcases}
			\Delta w=c_ms(y)\left(1-w^{\frac{4}{m-2}}\right)w & \hbox{on $M$}\\
			\partial_{\nu}w=d_mh(y)\left(w^{\frac{2}{m-2}}-1\right)w & \hbox{on $\partial M$}\,.
			\end{dcases}
			\eeqs
		Then, reasoning as for $u$, we conclude that $w\leq 1$.
		\end{proof}

\section*{Acknowledgements}
The authors thank the \emph{Departamento de Matem\'aticas} of the \emph{Universidad de Murcia}, where part of this paper has been written, for the warm hospitality. G.A. also thanks Lucio Mari for many fruitful discussions.

\end{document}